\documentclass[12pt]{article}
\usepackage{amsmath,amssymb,amsthm}
\usepackage{hyperref}
\usepackage{tikz}

\newtheorem{theorem}{Theorem}

\newtheorem{proposition}{Proposition}

\theoremstyle{definition}

\newcommand{\vertex}[2]{\node[ thick, fill=white, draw=black, circle, inner sep=2pt, shift={#2}] at #1 {};}

\begin{document}

\title{Using twins and scaling to construct cospectral graphs for the normalized Laplacian}

\author{
Steve Butler\thanks{Department of Mathematics, Iowa State University, Ames, IA 50011, USA
(butler@iastate.edu).  Partially supported by an NSA Young Investigator Grant.}}

\maketitle

\begin{abstract}
The spectrum of the normalized Laplacian matrix cannot determine the number of edges in a graph, however finding  constructions of cospectral graphs with differing number of edges has been elusive.  In this paper we use basic properties of twins and scaling to show how to construct such graphs.  We also give examples of families of graphs which are cospectral with a subgraph for the normalized Laplacian matrix.
\end{abstract}


\section{Introduction}
The goal in spectral graph theory is to understand the interplay between the structure of a graph and the eigenvalues of a matrix that is associated with the graph.  There are many possible ways to associate a matrix with a graph and each one has something different to say about the graph.  At the same time each matrix has limitations in what can be said about a graph, this is because of the existence of \emph{cospectral graphs} which are graphs which are not isomorphic but for which the corresponding matrices have the same eigenvalues.

One way to explore the limitations of a particular matrix in spectral graph theory is through the study of cospectral graphs and identifying a structural property that distinguish these graphs (such a property cannot then a priori be determined through the spectrum).  The four most common matrices that are studied are the adjacency matrix ($A$, where the $(i,j)$-entry indicates whether or not an edge is present), the combinatorial Laplacian ($L=D-A$, where $D$ is the diagonal matrix of the degrees $d(u)$), the signless Laplacian ($Q=D+A$) and the normalized Laplacian ($\mathcal{L}=D^{-1/2}(D-A)D^{-1/2}$ when the graph has no vertices of degree $0$; when the graph does have vertices of degree $0$ set the corresponding entry of $D^{-1/2}$ to $0$).  Information about these matrices can be obtained in several places including the work of Brouwer and Haemers \cite{brouwer} and Chung \cite{chung}.  

In Table~\ref{tab:weak}, taken from \cite{handbook}, for each matrix we indicate whether or not a particular type of matrix can detect a certain structural property.  Here ``Bip.'' stands for bipartite and ``Comp.'' stands for components.

\begin{table}[hftb]
\centering
\caption{Matrices and properties of a graph}
\label{tab:weak}
\begin{tabular}{|c||c|c|c|c|}\hline
Matrix&Bip.&\# Comp.&\# Bip.\ Comp.&\# Edges\\ \hline\hline
$A$&Yes&No&No&Yes\\ \hline
$L$&No&Yes&No&Yes\\ \hline
$Q$&No&No&Yes&Yes\\ \hline
$\mathcal{L}$&Yes&Yes&Yes&No\\ \hline
\end{tabular}
\end{table}

One of the striking things is that while the normalized Laplacian does well in detecting \emph{qualitative} structure (i.e., expansion \cite{chung}) it does poorly in some simple \emph{quantitative} values (i.e., counting the number of edges).  In some sense this follows because the normalized Laplacian matrix is up to a simple transformation related to the probability transition matrix of a random walk, which is given by $D^{-1}A$.  In particular we have that $D^{-1}A=D^{-1/2}(I-\mathcal{L})D^{1/2}$ and so two graphs are cospectral for the normalized Laplacian if and only if they are also cospectral with respect to the probability transition matrix for a random walk.  (This latter relationship is useful in checking cospectrality in some cases, and we will use it later in this paper.)

There have been several papers which have addressed constructions of cospectral graphs for the normalized Laplacian (see \cite{butler,grout,cavers,osborne}).  However none of these constructions yielded graphs with differing number of edges, and previously the only large graphs that were known which were cospectral and had differing number of edges were complete bipartite graphs.

The goal of this paper is to begin to address this situation by giving simple methods to construct cospectral graphs with respect to the normalized Laplacian matrix and which have differing number of edges.  This will be based on two aspects for the normalized Laplacian, namely scaling and twins (which we introduce in Section~\ref{sec:scaling} and then extend in Section~\ref{sec:twins}).  We also give examples of graphs which are cospectral with a subgraph (see Section~\ref{sec:subgraph}).  We then finish with some concluding remarks (see Section~\ref{sec:concluding}).

While our goal is to construct graphs which are simple, we will find it useful and informative to work with weighted graphs.  A weighted graph is a graph with an additional weight function on the edges, $w(u,v)$, which is non-negative and symmetric (i.e., $w(u,v)=w(v,u)$).  The adjacency matrix of a weighted graph is defined using the weight function by $A_{u,v}=w(u,v)$.  In Section~\ref{sec:twins} we will also need vertex weights which are nonnegative functions $w(u)$ (introduced for use in the normalized Laplacian in \cite{butler2}).  The degree of the vertex then becomes the sum of the vertex weight and the incident edges, i.e., $d(u)=w(u)+\sum_{v{\sim}u}w(u,v)$.  With this degree we can now define the diagonal degree matrix $D$ and along with $A$ can define $L$, $Q$ and $\mathcal{L}$.  We note that a vertex weight does not correspond to a loop, i.e., a loop affects both $D$ and $A$ while a vertex weight only affects $D$.  Simple graphs correspond to the situation when $w(u,v)\in\{0,1\}$ for all edges $u{\sim}v$ and $w(u)=0$ for all vertices $u$.

\section{Scaling and twins}\label{sec:scaling}
In this section we will introduce our two basic principles which when combined will allow us to construct cospectral graphs with differing number of edges.  The first one has to do with scaling and exploits the normalization aspect of the normalized Laplacian.

\begin{proposition}\label{prop:scaling}
Let $\alpha>0$, and let $G$ and $\alpha G$ be graphs on the same vertex set where  $w_{\alpha G}(u,v)=\alpha w_{G}(u,v)$ and $w_{\alpha G}(u)=\alpha w_{G}(u)$.  Then $\mathcal{L}_G=\mathcal{L}_{\alpha G}$, in particular the (weighted) graphs are cospectral with respect to the normalized Laplacian.
\end{proposition}
\begin{proof}
This follows because the degrees also scale by $\alpha$ and because of the definition of the normalized Laplacian, i.e., $\mathcal{L}=D^{-1/2}(D-A)D^{-1/2}$.  So we have
\begin{multline*}
\mathcal{L}_{\alpha G}=D_{\alpha G}^{-1/2}(D_{\alpha G}-A_{\alpha G})D_{\alpha G}^{-1/2}=\big(\frac{1}{\sqrt\alpha}D_G^{-1/2}\big)\big(\alpha(D_G-A_G)\big)\big(\frac{1}{\sqrt\alpha}D_G^{-1/2}\big)
\\=D_G^{-1/2}(D_G-A_G)D_G^{-1/2}=\mathcal{L}_G,
\end{multline*}
as desired.
\end{proof}

The second principle has to do with twin vertices.  In a simple graph twin vertices are two disjoint vertices that have the same set of neighbors.  For a weighted graph we will have a similar notion but we add the requirements that the edge weights scale.  Therefore we say $u$ and $v$ are \emph{twin vertices} if $u$ and $v$ are not adjacent, not isolated, and there is some $\alpha>0$ so that $\alpha w(u)=w(v)$, $\alpha w(u,u)= w(v,v)$, and for all $t\ne u,v$ we have $\alpha w(u,t)= w(v,t)$.  (This generalizes the notion of twins as introduced in \cite{handbook} and \cite{osborne}.)

As an example, let us consider the weighted graph shown in Figure~\ref{fig:vertextwins}.  All the vertex weights are zero, and the edge weights are $1$ unless otherwise marked (we will follow this convention throughout the paper).  Then we have that the vertices $a$ and $b$ are twins with the scaling factor $\alpha=2$ but that $a$ and $c$ are not twins because even though they have the same neighbors there is no consistent scaling factor for the edge weights.  In general we note that the twins relationship gives an equivalence relationship for the vertices of the graph, and moving forward the idea of grouping twins in these equivalence classes and then coalescing will be the idea behind our construction.  

\begin{figure}[htbf]
\centering
\begin{tikzpicture}[scale=0.55]
\draw[thick] (2,0)--(0,3)--(4,0)--(3,3)--(2,0)--(6,3)--(4,0)--(6,0);
\vertex{(2,0)}{(0,0)}
\vertex{(4,0)}{(0,0)}
\vertex{(6,0)}{(0,0)}
\vertex{(0,3)}{(0,0)}
\vertex{(3,3)}{(0,0)}
\vertex{(6,3)}{(0,0)}
\node at (0,3.5) {$a$};
\node at (3,3.5) {$b$};
\node at (6,3.5) {$c$};
\node[fill=white, inner sep=1.5pt] at (8/3,2) {\small $2$};
\node[fill=white, inner sep=1.5pt] at (10/3,2) {\small $2$};
\node[fill=white, inner sep=1.5pt] at (16/3,2) {\small $2$};
\end{tikzpicture}
\caption{Example of twin vertices}
\label{fig:vertextwins}
\end{figure}
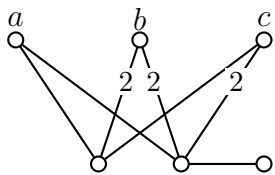

\begin{proposition}\label{prop:twins}
Let $G$ be a graph with twin vertices $u$ and $v$, let $\widehat{G}$ be a graph with $u$ and $v$ deleted and a new vertex $uv$ added where $w_{\widehat{G}}(uv,t)=w_G(u,t)+w_G(v,t)$ and $w_{\widehat{G}}(uv)=w_{G}(u)+w_G(v)$. Then the eigenvalues of the normalized Laplacian of $G$, counting multiplicity, are the eigenvalues of the normalized Laplacian of $\widehat{G}$ with an additional eigenvalue of $1$.
\end{proposition}

The proposition is a special case of a more general result given in the next section, so we will postpone the proof.  We also note that while we dealt with two twin vertices, we can repeatedly apply this result when we have many vertices which form an equivalence class of twin vertices.  In this case we simply combine the set of twin vertices into a single vertex and then add all the corresponding weights together, this will also create $1$ as an eigenvalue with multiplicity one less than the number of initial twins.

We are now ready to give our approach, which will be to form large bipartite graphs with many groups of equivalence classes of twin vertices, and then apply the previous propositions.  We will show that when properly done the resulting graphs are cospectral.

\begin{theorem}\label{thm:collapse}
Let $G$ and $H$ be (simple) graphs on $n$ vertices such that when all possible twin vertices are combined then the resulting weighted graphs $G'$ and $H'$ satisfy some isomorphism $H'=\alpha G'$ (i.e., by rescaling the weights then we can go from one graph to the other).  Then $G$ and $H$ are cospectral.
\end{theorem}
\begin{proof}
Since after reduction $G'$ and $H'$ have the same number of vertices then it must be that we had the same number of twin vertices to reduce in each graph, and by Proposition~\ref{prop:twins} each reduction contributed $1$ to the spectrum.  But the remaining eigenvalues, i.e., those determined by $G'$ and $H'$ also agree by Proposition~\ref{prop:scaling}.  Therefore the original graphs are cospectral.
\end{proof}

An example of Theorem~\ref{thm:collapse} is shown in Figure~\ref{fig:collapse1} where on the top row we show two simple graphs and then below them the corresponding (weighted) graphs when they have coalesced their twin vertices.  In particular the two coalesced graphs are, up to a scaling factor in the edge weights, the same and so are cospectral, and since in each case we reduced the number of vertices by $6$ we can conclude that the graphs in the first row are also cospectral.  Further it is easy to see that the graph on the left has $18$ edges while the graph on the right has $24$ edges.

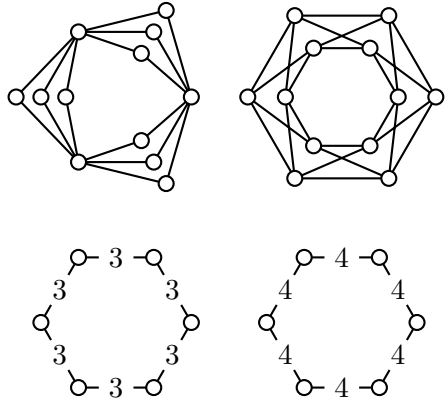
\begin{figure}[htbf]
\centering
\begin{tikzpicture}[scale=1]
\draw[thick] (0:1)--(60:.67)--(120:1)--(180:.67)--(240:1)--(300:.67)--(0:1)--(60:1)--(120:1)--(180:1)--(240:1)--(300:1)--(0:1)--(60:1.33)--(120:1)--(180:1.33)--(240:1)--(300:1.33)--cycle;
\vertex{(0:1)}{(0,0)}
\vertex{(60:0.67)}{(0,0)}
\vertex{(60:1)}{(0,0)}
\vertex{(60:1.33)}{(0,0)}
\vertex{(120:1)}{(0,0)}
\vertex{(180:0.67)}{(0,0)}
\vertex{(180:1)}{(0,0)}
\vertex{(180:1.33)}{(0,0)}
\vertex{(240:1)}{(0,0)}
\vertex{(300:0.67)}{(0,0)}
\vertex{(300:1)}{(0,0)}
\vertex{(300:1.33)}{(0,0)}

\draw[thick, shift={(0,-3)}] (0:1)--(60:1)--(120:1)--(180:1)--(240:1)--(300:1)--cycle;
\vertex{(0:1)}{(0,-3)}
\vertex{(60:1)}{(0,-3)}
\vertex{(120:1)}{(0,-3)}
\vertex{(180:1)}{(0,-3)}
\vertex{(240:1)}{(0,-3)}
\vertex{(300:1)}{(0,-3)}
\node[circle,fill=white,inner sep=1.5pt, shift={(0,-3)}] at (30:0.866) {\small $3$};
\node[circle,fill=white,inner sep=1.5pt, shift={(0,-3)}] at (90:0.866) {\small $3$};
\node[circle,fill=white,inner sep=1.5pt, shift={(0,-3)}] at (150:0.866) {\small $3$};
\node[circle,fill=white,inner sep=1.5pt, shift={(0,-3)}] at (210:0.866) {\small $3$};
\node[circle,fill=white,inner sep=1.5pt, shift={(0,-3)}] at (270:0.866) {\small $3$};
\node[circle,fill=white,inner sep=1.5pt, shift={(0,-3)}] at (330:0.866) {\small $3$};

\draw[thick,shift={(3,0)}] (0:0.75)--(60:0.75)--(120:0.75)--(180:0.75)--(240:0.75)--(300:0.75)--(0:0.75)--(60:1.25)--(120:0.75)--(180:1.25)--(240:0.75)--(300:1.25)--(0:1.25)--(60:0.75)--(120:1.25)--(180:0.75)--(240:1.25)--(300:0.75)--(0:1.25)--(60:1.25)--(120:1.25)--(180:1.25)--(240:1.25)--(300:1.25)--cycle;
\vertex{(0:0.75)}{(3,0)}
\vertex{(0:1.25)}{(3,0)}
\vertex{(60:0.75)}{(3,0)}
\vertex{(60:1.25)}{(3,0)}
\vertex{(120:0.75)}{(3,0)}
\vertex{(120:1.25)}{(3,0)}
\vertex{(180:0.75)}{(3,0)}
\vertex{(180:1.25)}{(3,0)}
\vertex{(240:0.75)}{(3,0)}
\vertex{(240:1.25)}{(3,0)}
\vertex{(300:0.75)}{(3,0)}
\vertex{(300:1.25)}{(3,0)}

\draw[thick, shift={(3,-3)}] (0:1)--(60:1)--(120:1)--(180:1)--(240:1)--(300:1)--cycle;
\vertex{(0:1)}{(3,-3)}
\vertex{(60:1)}{(3,-3)}
\vertex{(120:1)}{(3,-3)}
\vertex{(180:1)}{(3,-3)}
\vertex{(240:1)}{(3,-3)}
\vertex{(300:1)}{(3,-3)}
\node[circle,fill=white,inner sep=1.5pt, shift={(3,-3)}] at (30:0.866) {\small $4$};
\node[circle,fill=white,inner sep=1.5pt, shift={(3,-3)}] at (90:0.866) {\small $4$};
\node[circle,fill=white,inner sep=1.5pt, shift={(3,-3)}] at (150:0.866) {\small $4$};
\node[circle,fill=white,inner sep=1.5pt, shift={(3,-3)}] at (210:0.866) {\small $4$};
\node[circle,fill=white,inner sep=1.5pt, shift={(3,-3)}] at (270:0.866) {\small $4$};
\node[circle,fill=white,inner sep=1.5pt, shift={(3,-3)}] at (330:0.866) {\small $4$};
\end{tikzpicture}
\caption{An example of Theorem~\ref{thm:collapse}}
\label{fig:collapse1}
\end{figure}

This example can be generalized by fixing a bipartite graph and ``blowing'' it up in two different ways.  Suppose that $G$ is a bipartite graph with no isolated vertices and where the vertices have been partitioned as $A\cup B$ with $|A|=a$ and $|B|=b$.  Then for any natural numbers $s,s',t,t'$ such that $as+bt=as'+bt'$ we may construct the following cospectral pair. We construct $H$ by starting with $G$ and replacing each vertex in $A$ with $s$ independent vertices, each vertex in $B$ with $t$ independent vertices, and each edge in $G$ with  a copy of  $K_{s,t}$ connecting the corresponding vertices.  Similarly we  construct $H'$ by starting with $G$ and replacing each vertex in $A$ with $s'$ independent vertices, each vertex in $B$ with $t'$ independent vertices, and each edge in $G$ with a copy of $K_{s',t'}$. Both graphs have the same number of initial vertices (by assumption) and further upon coalescing the twin vertices we have introduced, the graphs are $stG$ and $s't'G$ which are cospectral and hence the initial graphs were also cospectral.  Also we note when $st\ne s't'$ then the corresponding graphs will have differing number of edges.

The graphs shown in Figure~\ref{fig:collapse1} are an example of this construction with $G$ being the $6$-cycle and $s=1$, $t=3$ and $s'=t'=2$.  Of course more interesting possibilities can occur.  In the top of Figure~\ref{fig:6cycle} we give an example of another way to build off of a $6$-cycle.  To simplify the picture we will mark the number of vertices that we blow up inside each vertex and then edges represent complete bipartite graphs.  Using the above techniques we first coalesce the twins coming from the blowup of $C_6$ (in total we reduce from $24$ to $6$ vertices in both cases; also note the new edge weight equals the product of the number of twins on each side of the blowup).  We now have the weighted graphs at the bottom of Figure~\ref{fig:6cycle} which differ by a scaling factor of $\alpha=\frac89$ and so can conclude that the original graphs were cospectral.

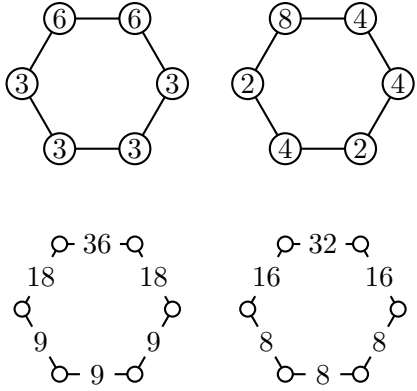
\begin{figure}[htbf]
\centering
\begin{tikzpicture}[scale=1]
\draw[thick] (0:1)--(60:1)--(120:1)--(180:1)--(240:1)--(300:1)--cycle;
\node[ thick, fill=white, draw=black, circle, inner sep=1pt, shift={(0,0)}] at (0:1) {\small $3$};
\node[ thick, fill=white, draw=black, circle, inner sep=1pt, shift={(0,0)}] at (60:1) {\small $6$};
\node[ thick, fill=white, draw=black, circle, inner sep=1pt, shift={(0,0)}] at (120:1) {\small $6$};
\node[ thick, fill=white, draw=black, circle, inner sep=1pt, shift={(0,0)}] at (180:1) {\small $3$};
\node[ thick, fill=white, draw=black, circle, inner sep=1pt, shift={(0,0)}] at (240:1) {\small $3$};
\node[ thick, fill=white, draw=black, circle, inner sep=1pt, shift={(0,0)}] at (300:1) {\small $3$};

\draw[thick, shift={(3,0)}] (0:1)--(60:1)--(120:1)--(180:1)--(240:1)--(300:1)--cycle;
\node[ thick, fill=white, draw=black, circle, inner sep=1pt, shift={(3,0)}] at (0:1) {\small $4$};
\node[ thick, fill=white, draw=black, circle, inner sep=1pt, shift={(3,0)}] at (60:1) {\small $4$};
\node[ thick, fill=white, draw=black, circle, inner sep=1pt, shift={(3,0)}] at (120:1) {\small $8$};
\node[ thick, fill=white, draw=black, circle, inner sep=1pt, shift={(3,0)}] at (180:1) {\small $2$};
\node[ thick, fill=white, draw=black, circle, inner sep=1pt, shift={(3,0)}] at (240:1) {\small $4$};
\node[ thick, fill=white, draw=black, circle, inner sep=1pt, shift={(3,0)}] at (300:1) {\small $2$};

\draw[thick, shift={(0,-3)}] (0:1)--(60:1)--(120:1)--(180:1)--(240:1)--(300:1)--cycle;
\vertex{(0:1)}{(0,-3)}
\vertex{(60:1)}{(0,-3)}
\vertex{(120:1)}{(0,-3)}
\vertex{(180:1)}{(0,-3)}
\vertex{(240:1)}{(0,-3)}
\vertex{(300:1)}{(0,-3)}
\node[circle,fill=white,inner sep=1.5pt, shift={(0,-3)}] at (30:0.866) {\small $18$};
\node[circle,fill=white,inner sep=1.5pt, shift={(0,-3)}] at (90:0.866) {\small $36$};
\node[circle,fill=white,inner sep=1.5pt, shift={(0,-3)}] at (150:0.866) {\small $18$};
\node[circle,fill=white,inner sep=1.5pt, shift={(0,-3)}] at (210:0.866) {\small $9$};
\node[circle,fill=white,inner sep=1.5pt, shift={(0,-3)}] at (270:0.866) {\small $9$};
\node[circle,fill=white,inner sep=1.5pt, shift={(0,-3)}] at (330:0.866) {\small $9$};

\draw[thick, shift={(3,-3)}] (0:1)--(60:1)--(120:1)--(180:1)--(240:1)--(300:1)--cycle;
\vertex{(0:1)}{(3,-3)}
\vertex{(60:1)}{(3,-3)}
\vertex{(120:1)}{(3,-3)}
\vertex{(180:1)}{(3,-3)}
\vertex{(240:1)}{(3,-3)}
\vertex{(300:1)}{(3,-3)}
\node[circle,fill=white,inner sep=1.5pt, shift={(3,-3)}] at (30:0.866) {\small $16$};
\node[circle,fill=white,inner sep=1.5pt, shift={(3,-3)}] at (90:0.866) {\small $32$};
\node[circle,fill=white,inner sep=1.5pt, shift={(3,-3)}] at (150:0.866) {\small $16$};
\node[circle,fill=white,inner sep=1.5pt, shift={(3,-3)}] at (210:0.866) {\small $8$};
\node[circle,fill=white,inner sep=1.5pt, shift={(3,-3)}] at (270:0.866) {\small $8$};
\node[circle,fill=white,inner sep=1.5pt, shift={(3,-3)}] at (330:0.866) {\small $8$};

\end{tikzpicture}
\caption{A more interesting blowup of the $6$-cycle}
\label{fig:6cycle}
\end{figure}

\section{Twin subgraphs}\label{sec:twins}
The preceding section looked at the idea of combining vertices together, the goal of this section is to look at ways to combine larger structures, which we will term \emph{twin subgraphs} together.  We will start by looking at the example of twin subgraphs shown in Figure~\ref{fig:twinsubgraphs}, where we have marked the twin subgraphs $H^{(1)}$ and $H^{(2)}$.  The basic idea is that these two subgraphs, along with how they connect to the remainder of the graph, agree up to a scaling factor.

\begin{figure}[htb]
\centering
\begin{tikzpicture}
\draw[thick] (2,0)--(0,0)--(0,1)--(1,1)--(1,-1)--(0,-1)--(0,0);
\vertex{(0,0)}{(0,0)}
\vertex{(0,1)}{(0,0)}
\vertex{(1,1)}{(0,0)}
\vertex{(1,0)}{(0,0)}
\vertex{(0,-1)}{(0,0)}
\vertex{(1,-1)}{(0,0)}
\vertex{(2,0)}{(0,0)}

\node[circle,fill=white,inner sep=1.5pt] at (0,-0.5) {\small $2$};
\node[circle,fill=white,inner sep=1.5pt] at (1,-0.5) {\small $2$};
\node[circle,fill=white,inner sep=1.5pt] at (0.5,-1) {\small $2$};

\node at (-0.25,1) {$a$};
\node at (-0.25,0) {$c$};
\node at (-0.25,-1) {$a'\,$};
\node at (1.25,1) {$b$};
\node at (1.25,0.25) {$d$};
\node at (1.25,-1) {$\,b'$};
\node at (2.25,0) {$e$};

\draw[rounded corners=4pt,dashed,color=black!50] (-0.5,0.75)--(-0.5,1.25)--(1.5,1.25)--(1.5,0.75)--cycle;
\draw[rounded corners=4pt,dashed,color=black!50] (-0.5,-0.75)--(-0.5,-1.25)--(1.5,-1.25)--(1.5,-0.75)--cycle;

\node at (2,1) {$H^{(1)}$};
\node at (2,-1) {$H^{(2)}$};
\end{tikzpicture}
\caption{Example of twin subgraphs}
\label{fig:twinsubgraphs}
\end{figure}
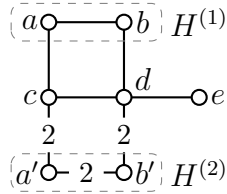

In general we say that $H^{(1)}$ and $H^{(2)}$ are \emph{twin subgraphs} of the graph $G$ if there is a partition of the vertices of $G$ as $V(G)=V_1\cup V_2\cup V_3$, a bijective map $\pi:V_1\to V_2$, and a fixed $\alpha>0$ which satisfy the following two properties:
\begin{itemize}
\item $G$ restricted to $V_1$ is $H^{(1)}$, $G$ restricted to $V_2$ is $H^{(2)}$, $w(u,v)=0$ for all $u\in V_1$ and $v\in V_2$, and $\alpha H^{(1)}= \pi(H^{(2)})$;
\item for each $u\in V_1$ and $t\in V_3$ we have $\alpha w(u,t)= w(\pi(u),t)$.
\end{itemize}

The first condition states that we have two graphs which are disjoint and not connected by an edge which agree up to some fixed scaling factor; the second condition states that the way that these two subgraphs connect with the remainder of the graph (i.e., vertices in $V_3$) also agree up to the same fixed scaling factor.

Applying this definition to the graph shown in Figure~\ref{fig:twinsubgraphs} we have that $V_1=\{a,b\}$, $V_2=\{a',b'\}$, $V_3=\{c,d,e\}$, $\pi(a)=a'$, $\pi(b)=b'$ and $\alpha=2$ so that $H^{(1)}$ and $H^{(2)}$ are twin subgraphs.  

Our goal is to combine the twins together and in the process reduce the size of the graph.  So given a graph $G$ with twin subgraphs $H^{(1)}$ and $H^{(2)}$ we will consider the following three graphs:
\begin{itemize}
\item $\widehat{G}$ formed by deleting the vertices $V_2$ and for each  $u\in V_1$ we set $w(u)=(1+\alpha)w(u)$ and for all $t\in V_1\cup V_3$ we set $w(u,t)=(1+\alpha)w(u,t)$;
\item $\widehat{H^{(i)}}$ formed by restricting $G$ to $H^{(i)}$ and then creating a new weight function $w'$ on the vertices by the following $w'(u)=w(u)+\sum_{t\in V_3}w(u,t)$.
\end{itemize}
The graph $\widehat{G}$ can be thought of the graph that is formed when $H^{(1)}$ and $H^{(2)}$ are combined together in $G$.  Note that by definition $\alpha \widehat{H^{(1)}}= \widehat{H^{(2)}}$ and so these graphs are cospectral and can be used interchangeably in the arguments below.  In Figure~\ref{fig:threegraphs} we give the three graphs which come from Figure~\ref{fig:twinsubgraphs}; this is our first time when vertex weights have come into play and we have marked these at a vertex by using a square box.

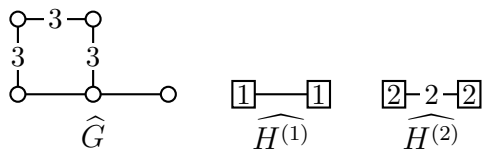
\begin{figure}[htb]
\centering
\begin{tikzpicture}
\draw[thick] (1,0)--(1,1)--(0,1)--(0,0)--(2,0);
\draw[thick] (3,0)--(4,0);
\draw[thick] (5,0)--(6,0);

\vertex{(0,0)}{(0,0)}
\vertex{(1,1)}{(0,0)}
\vertex{(1,0)}{(0,0)}
\vertex{(0,1)}{(0,0)}
\vertex{(2,0)}{(0,0)}

\node[fill=white,thick,draw,rectangle,inner sep=1.5pt] at (3,0) {\small$1$};
\node[fill=white,thick,draw,rectangle,inner sep=1.5pt] at (4,0) {\small$1$};
\node[fill=white,thick,draw,rectangle,inner sep=1.5pt] at (5,0) {\small$2$};
\node[fill=white,thick,draw,rectangle,inner sep=1.5pt] at (6,0) {\small$2$};
\node[fill=white,inner sep=1.5pt] at (5.5,0) {\small$2$};
\node[fill=white,inner sep=1.5pt] at (0.5,1) {\small$3$};
\node[fill=white,inner sep=1.5pt] at (0,0.5) {\small$3$};
\node[fill=white,inner sep=1.5pt] at (1,0.5) {\small$3$};

\node at (1,-0.5) {$\widehat{G}$};
\node at (3.5,-0.5) {$\widehat{H^{(1)}}$};
\node at (5.5,-0.5) {$\widehat{H^{(2)}}$};

\end{tikzpicture}
\caption{The three graphs coming from Figure~\ref{fig:twinsubgraphs}}
\label{fig:threegraphs}
\end{figure}

We are now ready to give our main result.

\begin{theorem}\label{thm:twins}
Suppose that $H^{(1)}$ and $H^{(2)}$ are twin subgraphs of $G$.  Then the eigenvalues (counting multiplicity) of the normalized Laplacian of $G$ is given by the union ot the eigenvalues of the normalized Laplacian of $\widehat{G}$ and the eigenvalues of the normalized Laplacian of $\widehat{H^{(1)}}$.
\end{theorem}

Before we begin the proof we introduce \emph{harmonic eigenvectors} for the normalized Laplacian matrix (see \cite{grout,chung}).  In particular, if $\mathcal{L}\mathbf{x}=\lambda\mathbf{x}$ then the harmonic eigenvector corresponding to $\lambda$ is $\mathbf{y}=D^{-1/2}\mathbf{x}$.  Note with this convention we have $(D-A)\mathbf{y}=\lambda D\mathbf{y}$, or rearranging, $A\mathbf{y}=(1-\lambda)D\mathbf{y}$.  So the  requirement for a harmonic eigenvector is that at each vertex $u$ the following is satisfied
\begin{equation}\label{eq:harmony}
\sum_{\substack{ v\\v{\sim}u}}w(u,v)\mathbf{y}(v)=(1-\lambda)d(u)\mathbf{y}(u).
\end{equation}
(Recall that $d(u)=w(u)+\sum_{u{\sim}v}w(u,v)$.)  Further, we have that two harmonic eigenvectors $\mathbf{y}$ and $\mathbf{z}$ are perpendicular if and only if the eigenvectors which they correspond to are perpendicular.  This translates to requiring that $D^{1/2}\mathbf{y}$ is perpendicular to $D^{1/2}\mathbf{z}$, or in other words
\[
0=(D^{1/2}\mathbf{y})^T(D^{1/2}\mathbf{z})=\mathbf{y}^TD\mathbf{z}=\sum_ud(u)\mathbf{y}(u)\mathbf{z}(u).
\]

\begin{proof}
We will show how to lift harmonic eigenvectors from $\widehat{G}$ and $\widehat{H^{(1)}}$ to harmonic eigenvectors of $G$ for the same corresponding eigenvalue.  Further the two families of harmonic eigenvectors thus produced are perpendicular and so by dimension arguments we will have found all of the eigenvalues of $G$ and produced the result.

So suppose that $\widehat{\mathbf{x}}$ is a harmonic eigenvector of $\widehat{G}$ for the eigenvalue $\lambda$.  Now consider the vector $\mathbf{x}$ defined as follows:
\[
\mathbf{x}(u)=\left\{\begin{array}{r@{\quad}l}
\widehat{\mathbf{x}}(u)&\text{if }u\in V_1\cup V_3,\\
\widehat{\mathbf{x}}(\pi^{-1}(u))&\text{if }u\in V_2.
\end{array}\right.
\]
We claim that $\mathbf{x}$ is a harmonic eigenvector of $G$ for the eigenvalue $\lambda$, to verify this we consider what happens at a vertex $u$.
\begin{itemize}
\item $u\in V_3$.  Note that for $v\in V_1$ that $w_{\widehat{G}}(u,v)=(1+\alpha)w_G(u,v)=w_G(u,v)+w_G(u,\pi(v))$ while for $v\in V_3$ that $w_{\widehat{G}}(u,v)=w_G(u,v)$.  Further a vertex in $V_3$ has the same degree both in $G$ and $\widehat{G}$.  Therefore we have
\[
\sum_{\substack{ v\\v{\sim}u}}w_G(u,v)\mathbf{x}(v)=
\sum_{\substack{ v\\v{\sim}u}}w_{\widehat{G}}(u,v)\widehat{\mathbf{x}}(v)=
(1-\lambda)d_{\widehat{G}}(u)\widehat{\mathbf{x}}(u)=
(1-\lambda)d_{G}(u)\mathbf{x}(u).
\]
\item $u\in V_1$.  Note that for all $v$ that $w_G(u,v)=\frac1{1+\alpha}w_{\widehat{G}}(u,v)$ and that $d_G(u)=\frac1{1+\alpha}d_{\widehat{G}}(u)$.  Therefore we have
\begin{multline*}
\sum_{\substack{ v\\v{\sim}u}}w_G(u,v)\mathbf{x}(v)=
\frac1{1+\alpha}\sum_{\substack{ v\\v{\sim}u}}w_{\widehat{G}}(u,v)\widehat{\mathbf{x}}(v)=
\frac1{1+\alpha}(1-\lambda)d_{\widehat{G}}(u)\widehat{\mathbf{x}}(u)\\=
(1-\lambda)d_{G}(u)\mathbf{x}(u).
\end{multline*}
\item $u\in V_2$.  This case follows similarly from what was done for the $u\in V_1$ case.
\end{itemize}

Now suppose that $\widehat{\mathbf{y}}$ is a harmonic eigenvector of $\widehat{H^{(1)}}$ for the eigenvalue $\lambda$.  Now consider the vector $\mathbf{y}$ defined as follows:
\[
\mathbf{y}(u)=\left\{\begin{array}{r@{\quad}l}
\alpha \widehat{\mathbf{y}}(u)&\text{if }u\in V_1,\\
-\widehat{\mathbf{y}}(\pi^{-1}(u))&\text{if }u\in V_2,\\
0&\text{else}.
\end{array}\right.
\]
We claim that $\mathbf{y}$ is a harmonic eigenvector of $G$ for the eigenvalue $\lambda$, to verify this we consider what happens at a vertex $u$.
\begin{itemize}
\item $u\in V_3$.  By definition for $v\in V_1$ we have that $\alpha w_G(u,v)=w_G(u,\pi(v))$.  Therefore we can conclude
\begin{multline*}
w_G(u,v)\mathbf{y}(v)+w_G(u,\pi(v))\mathbf{y}(\pi(v))=
w_G(u,v)\mathbf{y}(v)+\alpha w_G(u,v)\mathbf{y}(\pi(v))\\=
w_G(u,v)(\alpha \widehat{\mathbf{y}}(v))+\alpha w_G(u,v)(-\widehat{\mathbf{y}}(v))=0.
\end{multline*}
Using this, and that all vertices in $V_3$ (including $u$) are $0$ for $\mathbf{y}$ we can conclude
\[
\sum_{\substack{ v\\v{\sim}u}}w_G(u,v)\mathbf{y}(v)=0=(1-\lambda)d_{G}(u)\mathbf{y}(u)
\]
\item $u\in V_1$.  For a vertex $u\in V_1$ we have that for all $v\in V_1\cup V_3$ $w_G(u,v)=\frac1{1+\alpha}w_{\widehat{G}}(u,v)$ (by definition there are no edges between $V_1$ and $V_2$) and that $d_G(u)=\frac{1}{1+\alpha}d_{\widehat{G}}(u)$.  Therefore we have
\begin{multline*}
\sum_{\substack{ v\\v{\sim}u}}w_G(u,v)\mathbf{y}(v)=
\frac\alpha{1+\alpha}\sum_{\substack{ v\\v{\sim}u}}w_{\widehat{G}}(u,v)\widehat{\mathbf{y}}(v)=
\frac\alpha{1+\alpha}(1-\lambda)d_{\widehat{G}}\widehat{\mathbf{y}}(u)\\
=(1-\lambda)d_G\mathbf{y}(u).
\end{multline*}
\item $u\in V_2$.  This case follows similarly from what was done for the $u\in V_1$ case.
\end{itemize}

Finally we  verify that the resulting vectors are orthogonal.  Now suppose that  $\mathbf{x}$ and $\mathbf{y}$ were harmonic eigenvectors lifted from $\widehat{G}$ and $\widehat{H^{(1)}}$ respectively.  We have $d(u)\mathbf{x}(u)\mathbf{y}(u)$ is trivially $0$ for all $u\in V_3$ (from $\mathbf{y}(u)=0$).  Further we have for $u\in V_1$ that
\[
d(u)\mathbf{x}(u)\mathbf{y}(u)+d(\pi(u))\mathbf{x}(\pi(u))\mathbf{y}(\pi(u))=
d(u)\mathbf{x}(u)\mathbf{y}(u)+\alpha d(u)\mathbf{x}(u)\bigg(-\frac{1}{\alpha}\mathbf{y}(u)\bigg)=0.
\]
This shows the remaining terms can be paired to cancel and these harmonic eigenvectors are orthogonal, which concludes the proof.
\end{proof}

We can now give a proof for the case of twin vertices that came from the preceding section.

\begin{proof}[Proof of Proposition~\ref{prop:twins}]
Let $V_1=\{u\}$, $V_2=\{v\}$, and $V_3$ the remaining vertices.  Then by Theorem~\ref{thm:twins} we now have that the eigenvalues of $G$ are the eigenvalues of $\widehat{G}$ together with the eigenvalues of a graph consisting of a single vertex with no loops and positive vertex weight (i.e., since the vertex twins were not isolated).  For this latter graph we have that $\mathcal{L}=D^{-1/2}(D-A)D^{-1/2}=[1]$ which has eigenvalue $1$.  The result now follows.
\end{proof}

As in the previous section we can now combine twin subgraphs to form cospectral graphs, we now just have a more interesting array of graphs to work with.  A special case of this phenomenon was previously given by Osborne \cite{osborne} wherein the twin subgraphs were $sP_{n+1}$ and $tP_{n+1}$ and they were glued together at a vertex.  In particular Osborne showed that the spectrum was independent of the choice of $s$ and $t$ and then used this to construct cospectral simple graphs.  

Of course several twin subgraphs can be glued together and the theorem still applies.  An example of this is shown in the top of Figure~\ref{fig:triplets} (again to save space we  mark the number of twins inside of each vertex if it corresponds to multiple twins).  We first start by combining vertex twins (reducing by $16$ vertices in each case) to get the graphs shown in the bottom of Figure~\ref{fig:triplets}.  Now in both cases we have three subgraphs which are mutual twins gluing together at a central vertex.  We can combine these graphs together to conclude (after appropriate scaling) that the remaining eigenvalues come from a path of length $4$ and two copies of a path of length $3$ with one end of the path having a vertex weight of $1$.  This example shows it is possible to iteratively combine twins to form cospectral graphs.  

\begin{figure}[hftb]
\centering
\begin{tikzpicture}[scale=1]
\draw[thick] (0,0)--(90:2.4);
\draw[thick] (0,0)--(-30:2.4);
\draw[thick] (0,0)--(210:2.4);
\draw[thick, shift={(6,0)}] (0,0)--(90:2.4);
\draw[thick, shift={(6,0)}] (0,0)--(-30:2.4);
\draw[thick, shift={(6,0)}] (0,0)--(210:2.4);
\vertex{(0,0)}{(0,0)}
\node[thick, fill=white, draw=black, circle, inner sep=1pt, shift={(0,0)}] at (90:.8) {\small $2$};
\vertex{(90:1.6)}{(0,0)}
\node[thick, fill=white, draw=black, circle, inner sep=1pt, shift={(0,0)}] at (90:2.4) {\small $2$};
\node[thick, fill=white, draw=black, circle, inner sep=1pt, shift={(0,0)}] at (-30:.8) {\small $4$};
\vertex{(-30:1.6)}{(0,0)}
\node[thick, fill=white, draw=black, circle, inner sep=1pt, shift={(0,0)}] at (-30:2.4) {\small $4$};
\node[thick, fill=white, draw=black, circle, inner sep=1pt, shift={(0,0)}] at (210:.8) {\small $5$};
\vertex{(210:1.6)}{(0,0)}
\node[thick, fill=white, draw=black, circle, inner sep=1pt, shift={(0,0)}] at (210:2.4) {\small $5$};
\vertex{(0,0)}{(6,0)}
\vertex{(90:.8)}{(6,0)}
\vertex{(90:1.6)}{(6,0)}
\vertex{(90:2.4)}{(6,0)}
\node[thick, fill=white, draw=black, circle, inner sep=1pt, shift={(6,0)}] at (-30:.8) {\small $3$};
\vertex{(-30:1.6)}{(6,0)}
\node[thick, fill=white, draw=black, circle, inner sep=1pt, shift={(6,0)}] at (-30:2.4) {\small $3$};
\node[thick, fill=white, draw=black, circle, inner sep=1pt, shift={(6,0)}] at (210:.8) {\small $7$};
\vertex{(210:1.6)}{(6,0)}
\node[thick, fill=white, draw=black, circle, inner sep=1pt, shift={(6,0)}] at (210:2.4) {\small $7$};
\end{tikzpicture}
\bigskip
\bigskip
\begin{tikzpicture}[scale=1]
\draw[thick] (0,0)--(90:2.4);
\draw[thick] (0,0)--(-30:2.4);
\draw[thick] (0,0)--(210:2.4);
\draw[thick, shift={(6,0)}] (0,0)--(90:2.4);
\draw[thick, shift={(6,0)}] (0,0)--(-30:2.4);
\draw[thick, shift={(6,0)}] (0,0)--(210:2.4);
\vertex{(0,0)}{(0,0)}
\vertex{(90:0.8)}{(0,0)}
\vertex{(90:1.6)}{(0,0)}
\vertex{(90:2.4)}{(0,0)}
\vertex{(210:0.8)}{(0,0)}
\vertex{(210:1.6)}{(0,0)}
\vertex{(210:2.4)}{(0,0)}
\vertex{(-30:0.8)}{(0,0)}
\vertex{(-30:1.6)}{(0,0)}
\vertex{(-30:2.4)}{(0,0)}
\vertex{(0,0)}{(6,0)}
\vertex{(90:0.8)}{(6,0)}
\vertex{(90:1.6)}{(6,0)}
\vertex{(90:2.4)}{(6,0)}
\vertex{(210:0.8)}{(6,0)}
\vertex{(210:1.6)}{(6,0)}
\vertex{(210:2.4)}{(6,0)}
\vertex{(-30:0.8)}{(6,0)}
\vertex{(-30:1.6)}{(6,0)}
\vertex{(-30:2.4)}{(6,0)}

\node[fill=white, inner sep=1.5pt] at (90:0.4) {\small$2$};
\node[fill=white, inner sep=1.5pt] at (90:1.2) {\small$2$};
\node[fill=white, inner sep=1.5pt] at (90:2.0) {\small$2$};
\node[fill=white, inner sep=1.5pt] at (210:0.4) {\small$5$};
\node[fill=white, inner sep=1.5pt] at (210:1.2) {\small$5$};
\node[fill=white, inner sep=1.5pt] at (210:2.0) {\small$5$};
\node[fill=white, inner sep=1.5pt] at (-30:0.4) {\small$4$};
\node[fill=white, inner sep=1.5pt] at (-30:1.2) {\small$4$};
\node[fill=white, inner sep=1.5pt] at (-30:2.0) {\small$4$};

\node[fill=white, inner sep=1.5pt, shift={(6,0)}] at (210:0.4) {\small$7$};
\node[fill=white, inner sep=1.5pt, shift={(6,0)}] at (210:1.2) {\small$7$};
\node[fill=white, inner sep=1.5pt, shift={(6,0)}] at (210:2.0) {\small$7$};
\node[fill=white, inner sep=1.5pt, shift={(6,0)}] at (-30:0.4) {\small$3$};
\node[fill=white, inner sep=1.5pt, shift={(6,0)}] at (-30:1.2) {\small$3$};
\node[fill=white, inner sep=1.5pt, shift={(6,0)}] at (-30:2.0) {\small$3$};

\end{tikzpicture}
\caption{Two cospectral graphs arising from having triplet subgraphs}
\label{fig:triplets}
\end{figure}
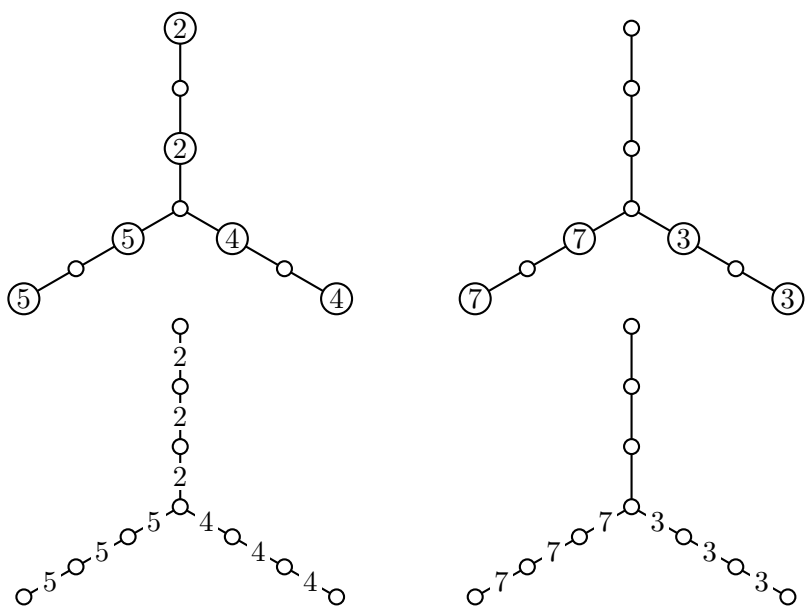

We note that Theorem~\ref{thm:twins} shows we could connect the same arbitrary subgraph to the center vertices of the pair of graphs in Figure~\ref{fig:triplets} and the resulting graphs would still be cospectral.  A very elementary form of this method for constructing cospectral graphs was given in Butler and Grout \cite{grout} for constructing non-isomorphic cospectral graphs for the normalized Laplacian, where the ``rays'' from the central vertex only went out to distance two.

The examples mentioned so far rely on gluing graphs at one vertex, but other configurations can arise.  In the next section we will give an example of twin subgraphs where the gluing happens at two distinct vertices.

\section{Graphs cospectral with their subgraph}\label{sec:subgraph}
Since we can construct graphs which are cospectral and have differing number of edges we have the possibility that there are graphs which are cospectral with one of their subgraphs for the normalized Laplacian matrix.  A handful of examples of this type were given by Butler and Grout \cite{grout}, and here we generalize two of these to form an infinite sequence of examples of this behavior.  These graphs also show that it is possible for a dense graph to be cospectral with a sparse subgraph.

Our first construction is shown in the top of Figure~\ref{fig:subgraph1} where we again follow the convention of marking any twins by indicating that with a label at the vertex where we want to take twin copies.  In particular the difference between the two graphs is a removal of a $K_{k,k}$ represented by the edge on the bottom. 

\begin{figure}[hftb]
\centering
\begin{tikzpicture}[scale=1.25]
\draw[thick] (0,2)--(0,0)--(2,0)--(2,2)--(0,2)--(1,3)--(2,2);
\draw[thick] (6,0)--(6,2)--(8,2)--(8,0);
\draw[thick] (6,2)--(7,3)--(8,2);
\node[thick,fill=white,draw=black,circle,inner sep=1pt,shift={(0,0)}] at (0,0) {\large $k$};
\node[thick,fill=white,draw=black,circle,inner sep=1pt,shift={(0,0)}] at (2,0) {\large $k$};
\vertex{(2,2)}{(0,0)}
\vertex{(0,2)}{(0,0)}
\node[thick,fill=white,draw=black,circle,inner sep=1pt,shift={(0,0)}] at (1,3) {\scriptsize $k+1$};
\node[thick,fill=white,draw=black,circle,inner sep=1pt,shift={(0,0)}] at (6,0) {\large $k$};
\node[thick,fill=white,draw=black,circle,inner sep=1pt,shift={(0,0)}] at (8,0) {\large $k$};
\vertex{(8,2)}{(0,0)}
\vertex{(6,2)}{(0,0)}
\node[thick,fill=white,draw=black,circle,inner sep=1pt,shift={(0,0)}] at (7,3) {\scriptsize $k+1$};
\end{tikzpicture}

\bigskip\bigskip

\begin{tikzpicture}[scale=1.25]
\draw[thick] (0,2)--(0,0)--(2,0)--(2,2)--(0,2)--(1,3)--(2,2);
\draw[thick, shift={(0,0)}] (6,0)--(6,2)--(8,2)--(8,0);
\draw[thick, shift={(0,0)}] (6,2)--(7,3)--(8,2);
\vertex{(0,0)}{(0,0)}
\vertex{(0,2)}{(0,0)}
\vertex{(2,2)}{(0,0)}
\vertex{(2,0)}{(0,0)}
\vertex{(1,3)}{(0,0)}
\vertex{(6,0)}{(0,0)}
\vertex{(6,2)}{(0,0)}
\vertex{(8,2)}{(0,0)}
\vertex{(8,0)}{(0,0)}
\vertex{(7,3)}{(0,0)}
\node[rectangle, fill=white, inner sep=2pt] at (0,1) {\small $k$};
\node[rectangle, fill=white, inner sep=2pt] at (2,1) {\small $k$};
\node[rectangle, fill=white, inner sep=1pt] at (0.5,2.5) {\small $k+1$};
\node[rectangle, fill=white, inner sep=1pt] at (1.5,2.5) {\small $k+1$};
\node[rectangle, fill=white, inner sep=2pt] at (1,2) {\small $1$};
\node[rectangle, fill=white, inner sep=2pt] at (1,0) {\small $k^2$};
\node[rectangle, fill=white, inner sep=2pt, shift={(0,0)}] at (6,1) {\small $k$};
\node[rectangle, fill=white, inner sep=2pt, shift={(0,0)}] at (8,1) {\small $k$};
\node[rectangle, fill=white, inner sep=1pt, shift={(0,0)}] at (6.5,2.5) {\small $k+1$};
\node[rectangle, fill=white, inner sep=1pt, shift={(0,0)}] at (7.5,2.5) {\small $k+1$};
\node[rectangle, fill=white, inner sep=2pt, shift={(0,0)}] at (7,2) {\small $1$};
\end{tikzpicture}
\caption{A pair of cospectral graphs, one a subgraph of the other, and their coalescing}
\label{fig:subgraph1}
\end{figure}
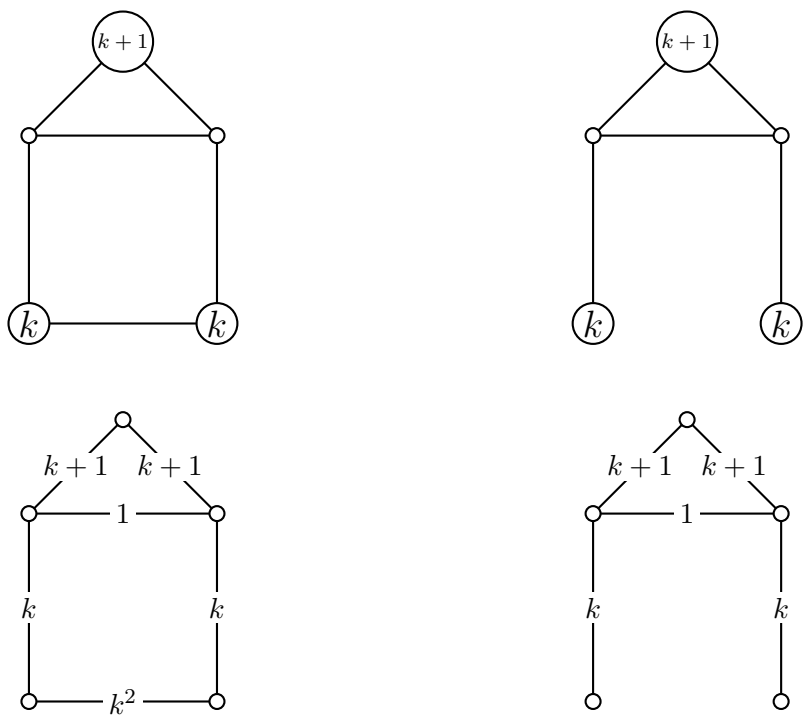

For both graphs in Figure~\ref{fig:subgraph1} we have removed $3k-2$ vertices when coalescing so that the spectrum will consist of $1$ with multiplicity $3k-2$ in addition to the spectrum of the coalesced graphs shown at the bottom of Figure~\ref{fig:subgraph1}.  These are graphs on 5 vertices and so we can write down explicitly the normalized Laplacian for these two matrices, but as already noted the matrices are a transformation of the probability transition matrix (i.e., $D^{-1}A$) and so we instead write these down here (where the vertices are labeled going from top to bottom and left to right).
\[
\left[\begin{array}{ccccc}
0&\frac12&\frac12&0&0\\
\frac12&0&\frac1{2k+2}&\frac{k}{2k+2}&0\\
\frac12&\frac1{2k+2}&0&0&\frac{k}{2k+2}\\
0&\frac1{k+1}&0&0&\frac{k}{k+1}\\
0&0&\frac1{k+1}&\frac{k}{k+1}&0
\end{array}\right]\qquad\qquad
\left[\begin{array}{ccccc}
0&\frac12&\frac12&0&0\\
\frac12&0&\frac1{2k+2}&\frac{k}{2k+2}&0\\
\frac12&\frac1{2k+2}&0&0&\frac{k}{2k+2}\\
0&1&0&0&0\\
0&0&1&0&0
\end{array}\right]
\]
The characteristic polynomial for \emph{both} of these matrices is
\[
x^5-\frac{6k^2+8k+3}{4(k+1)^2}x^3-\frac{1}{4(k+1)}x^2+\frac{k(2k+1)}{4(k+1)^2}x.
\]
In particular the remaining five eigenvalues to determine for the original graph must also agree, showing that the graphs are cospectral.

The second construction is similar but instead of only having twin vertices we also have twin subgraphs.  An example of a pair of such graphs for the special case $k=2$ is shown in Figure~\ref{fig:examplepairs}.  The general construction is shown at the top in Figure~\ref{fig:subgraph2} where we take the indicated number of copies of the ``edge'' inside of the dashed rectangle.  By coalescing the various twins, equal in number and type in both graphs, we reduce down to the $6$ vertex graphs with appropriate edge weights shown at the bottom of Figure~\ref{fig:subgraph2}.

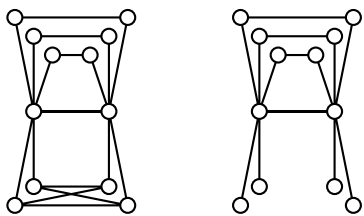
\begin{figure}[hftb]
\centering
\begin{tikzpicture}
\draw[thick] (0,0)--(1,0)--(1,1)--(0,1)--(0,0)--(1,0)--(0.75,0.75)--(0.25,0.75)--(0,0)--(1,0)--(1.25,1.25)--(-0.25,1.25)--(0,0);
\draw[thick] (-0.25,-1.25)--(0,0)--(0,-1);
\draw[thick] (1.25,-1.25)--(1,0)--(1,-1);
\draw[thick] (-0.25,-1.25)--(1.25,-1.25)--(0,-1)--(1,-1)--cycle;
\vertex{(0,0)}{(0,0)}
\vertex{(1,0)}{(0,0)}
\vertex{(1,1)}{(0,0)}
\vertex{(0,1)}{(0,0)}
\vertex{(0.75,0.75)}{(0,0)}
\vertex{(0.25,0.75)}{(0,0)}
\vertex{(1.25,1.25)}{(0,0)}
\vertex{(-.25,1.25)}{(0,0)}
\vertex{(1,-1)}{(0,0)}
\vertex{(0,-1)}{(0,0)}
\vertex{(-.25,-1.25)}{(0,0)}
\vertex{(1.25,-1.25)}{(0,0)}

\draw[thick, shift={(3,0)}] (0,0)--(1,0)--(1,1)--(0,1)--(0,0)--(1,0)--(0.75,0.75)--(0.25,0.75)--(0,0)--(1,0)--(1.25,1.25)--(-0.25,1.25)--(0,0);
\draw[thick, shift={(3,0)}] (-0.25,-1.25)--(0,0)--(0,-1);
\draw[thick, shift={(3,0)}] (1.25,-1.25)--(1,0)--(1,-1);
\vertex{(0,0)}{(3,0)}
\vertex{(1,0)}{(3,0)}
\vertex{(1,1)}{(3,0)}
\vertex{(0,1)}{(3,0)}
\vertex{(0.75,0.75)}{(3,0)}
\vertex{(0.25,0.75)}{(3,0)}
\vertex{(1.25,1.25)}{(3,0)}
\vertex{(-.25,1.25)}{(3,0)}
\vertex{(1,-1)}{(3,0)}
\vertex{(0,-1)}{(3,0)}
\vertex{(-.25,-1.25)}{(3,0)}
\vertex{(1.25,-1.25)}{(3,0)}
\end{tikzpicture}
\caption{An example of cospectral graphs with one a subgraph of the other}
\label{fig:examplepairs}
\end{figure}

\begin{figure}[hftb]
\centering
\begin{tikzpicture}[scale=1]
\draw[thick] (0,2)--(0,0)--(2,0)--(2,2)--(0,2)--(0,4)--(2,4)--(2,2);
\draw[thick, shift={(0,0)}] (6,0)--(6,2)--(8,2)--(8,0);
\draw[thick, shift={(0,0)}] (6,2)--(6,4)--(8,4)--(8,2);
\draw[thick, rounded corners, dashed] (-0.25,3.75)--(2.25,3.75)--(2.25,4.25)--(-0.25,4.25)--cycle;
\draw[thick, rounded corners, dashed, shift={(6,0)}] (-0.25,3.75)--(2.25,3.75)--(2.25,4.25)--(-0.25,4.25)--cycle;
\node[] at (1,4.5) {$\times(k+1)$};
\node[shift={(0,0)}] at (7,4.5) {$\times(k+1)$};
\vertex{(0,4)}{(0,0)}
\vertex{(2,4)}{(0,0)}
\vertex{(6,4)}{(0,0)}
\vertex{(8,4)}{(0,0)}
\node[thick,fill=white,draw=black,circle,inner sep=1pt,shift={(0,0)}] at (0,0) {\large $k$};
\node[thick,fill=white,draw=black,circle,inner sep=1pt,shift={(0,0)}] at (2,0) {\large $k$};
\vertex{(2,2)}{(0,0)}
\vertex{(0,2)}{(0,0)}
\node[thick,fill=white,draw=black,circle,inner sep=1pt,shift={(0,0)}] at (6,0) {\large $k$};
\node[thick,fill=white,draw=black,circle,inner sep=1pt,shift={(0,0)}] at (8,0) {\large $k$};
\vertex{(8,2)}{(0,0)}
\vertex{(6,2)}{(0,0)}
\end{tikzpicture}

\bigskip\bigskip

\begin{tikzpicture}[scale=1]
\draw[thick] (0,2)--(0,0)--(2,0)--(2,2)--(0,2)--(0,4)--(2,4)--(2,2);
\draw[thick, shift={(0,0)}] (6,0)--(6,2)--(8,2)--(8,0);
\draw[thick, shift={(0,0)}] (6,2)--(6,4)--(8,4)--(8,2);
\vertex{(0,0)}{(0,0)}
\vertex{(0,2)}{(0,0)}
\vertex{(2,2)}{(0,0)}
\vertex{(2,0)}{(0,0)}
\vertex{(0,4)}{(0,0)}
\vertex{(2,4)}{(0,0)}
\vertex{(6,0)}{(0,0)}
\vertex{(6,2)}{(0,0)}
\vertex{(8,2)}{(0,0)}
\vertex{(8,0)}{(0,0)}
\vertex{(6,4)}{(0,0)}
\vertex{(8,4)}{(0,0)}
\node[rectangle, fill=white,  inner sep=2pt] at (0,1) {\small $k$};
\node[rectangle, fill=white,  inner sep=2pt] at (2,1) {\small $k$};
\node[rectangle, fill=white,  inner sep=1pt] at (0,3) {\small $k+1$};
\node[rectangle, fill=white,  inner sep=1pt] at (1,4) {\small $k+1$};
\node[rectangle, fill=white,  inner sep=1pt] at (2,3) {\small $k+1$};
\node[rectangle, fill=white,  inner sep=2pt] at (1,2) {\small $1$};
\node[rectangle, fill=white,  inner sep=2pt] at (1,0) {\small $k^2$};
\node[rectangle, fill=white,  inner sep=2pt, shift={(0,0)}] at (6,1) {\small $k$};
\node[rectangle, fill=white,  inner sep=2pt, shift={(0,0)}] at (8,1) {\small $k$};
\node[rectangle, fill=white,  inner sep=1pt, shift={(0,0)}] at (6,3) {\small $k+1$};
\node[rectangle, fill=white,  inner sep=1pt, shift={(0,0)}] at (7,4) {\small $k+1$};
\node[rectangle, fill=white,  inner sep=1pt, shift={(0,0)}] at (8,3) {\small $k+1$};
\node[rectangle, fill=white, inner sep=2pt, shift={(0,0)}] at (7,2) {\small $1$};
\end{tikzpicture}
\caption{A pair of cospectral graphs, one a subgraph of the other, and their coalescing}
\label{fig:subgraph2}
\end{figure}
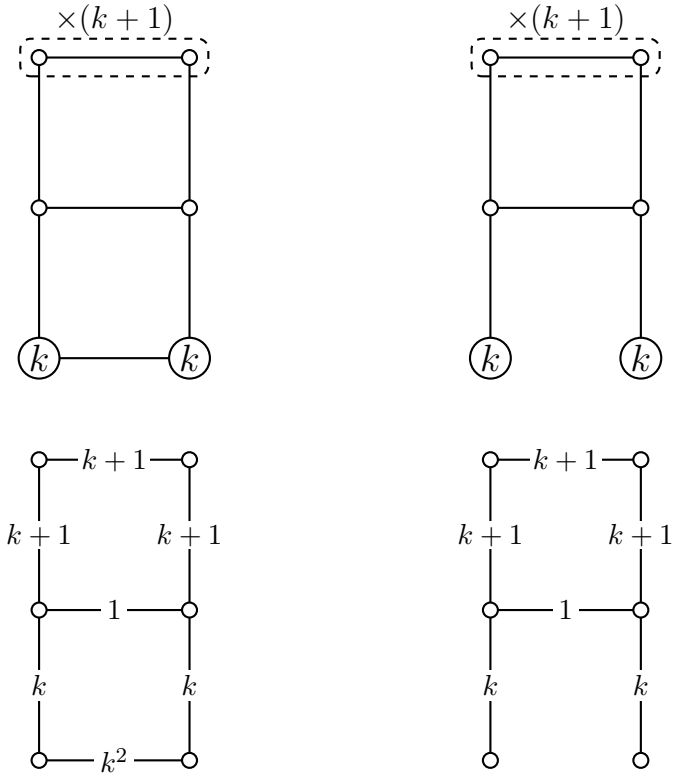

Again we only need to check that the graphs on the bottom of Figure~\ref{fig:subgraph2} are cospectral.  As before we write down the probability transition matrices, which are as follows:
\[
\left[\begin{array}{cccccc}
0&\frac12&\frac12&0&0&0\\
\frac12&0&0&\frac12&0&0\\
\frac12&0&0&\frac1{2k+2}&\frac{k}{2k+2}&0\\
0&\frac12&\frac1{2k+2}&0&0&\frac{k}{2k+2}\\
0&0&\frac1{k+1}&0&0&\frac{k}{k+1}\\
0&0&0&\frac1{k+1}&\frac{k}{k+1}&0
\end{array}\right]\qquad
\left[\begin{array}{cccccc}
0&\frac12&\frac12&0&0&0\\
\frac12&0&0&\frac12&0&0\\
\frac12&0&0&\frac1{2k+2}&\frac{k}{2k+2}&0\\
0&\frac12&\frac1{2k+2}&0&0&\frac{k}{2k+2}\\
0&1&0&0&0\\
0&0&1&0&0
\end{array}\right]
\]
The characteristic polynomial for \emph{both} of these matrices is
\[
x^6-\frac{7k^2+10k+4}{4(k+1)^2}x^4+\frac{k(13k+8)}{16(k+1)^2}x^2-\frac{k^2}{16(k+1)^2}.
\]
Thus these graphs are also cospectral.

\section{Concluding remarks}\label{sec:concluding}
In this paper we have looked at how we can use the properties of scaling and  twins in graphs to produce some examples of cospectral graphs.  In particular, we have seen ways to construct cospectral graphs with differing number of edges for the normalized Laplacian matrix.

While this covers many basic and beautiful cases, there are still many pairs of graphs which are known to be cospectral but for which no current explanation is known (beyond coincidence).  Further progress in understanding these structures will be interesting.  One special case that warrants further exploration is when a graph is cospectral with its subgraph.  The examples given here show that we can pull out ``most'' of the edges of a special graph and not change the spectrum.  However in all known cases when a graph is cospectral with a subgraph, the graph that is pulled out is a bipartite graph.  It would be interesting to find a graph that is cospectral with its subgraph when we pull out a non-bipartite graph or show no such pairing exists.

As  noted in the table at the beginning we cannot form graphs which are cospectral for the adjacency matrix and have differing number of edges.  We can however still coalesce twin vertices, so the obstacle for the approach outlined here is in the scaling. In particular for the adjacency matrix the graph $\alpha G$ is not cospectral with $G$, but rather the eigenvalues all scale by $\alpha$.  This suggests a new notion, namely we say that $G$ and $H$ are \emph{$\alpha$-cospectral with respect to the adjacency matrix} if there is some $\alpha>0$ so that the  eigenvalues of $H$, counting multiplicity, is obtained by scaling  each of the eigenvalues of $G$ by the value $\alpha$.  With this convention it can be shown that the construction of bipartite graphs in Section~\ref{sec:scaling} produce $\alpha$-cospectral graphs. 

By computer experiment, most small graphs which are $\alpha$-cospectral with respect to the adjacency matrix tend to have twin vertices, and as a result tend to have $0$ as an eigenvalue, often with high multiplicity.  However this is not required, and in Figure~\ref{fig:alpha} we give an example of two graphs which are $\alpha$-cospectral with respect to the adjacency matrix but have no eigenvalue of $0$.  (The spectrum for the graph on the left is $-\sqrt{2}^{(4)}, \sqrt{2}^{(4)}, \pm1\pm\sqrt3$ while the spectrum for the graph on the right is $-1^{(4)}, 1^{(4)}, \pm\sqrt{2\pm\sqrt3}$ where exponents indicate multiplicity; in this case we have $\alpha=\sqrt2$.)  It would be interesting to understand some other constructions of $\alpha$-cospectral graphs as well as a more general theory about what $\alpha$-cospectral graphs have in common.

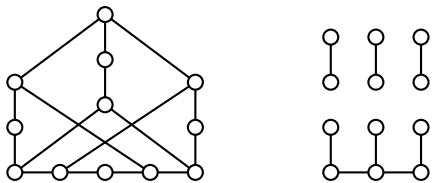
\begin{figure}[hftb]
\centering
\begin{tikzpicture}[scale=0.6]
\draw[thick] (0,0)--(4,0)--(4,2)--(2,3.5)--(0,2)--(0,0)--(2,1.5)--(4,0);
\draw[thick] (0,2)--(3,0);
\draw[thick] (1,0)--(4,2);
\draw[thick] (2,1.5)--(2,3.5);
\vertex{(0,2)}{(0,0)}
\vertex{(0,1)}{(0,0)}
\vertex{(0,0)}{(0,0)}
\vertex{(1,0)}{(0,0)}
\vertex{(2,0)}{(0,0)}
\vertex{(3,0)}{(0,0)}
\vertex{(4,0)}{(0,0)}
\vertex{(4,1)}{(0,0)}
\vertex{(4,2)}{(0,0)}
\vertex{(2,3.5)}{(0,0)}
\vertex{(2,2.5)}{(0,0)}
\vertex{(2,1.5)}{(0,0)}

\draw[thick] (7,1)--(7,0)--(9,0)--(9,1);
\draw[thick] (8,0)--(8,1);
\draw[thick] (7,2)--(7,3);
\draw[thick] (8,2)--(8,3);
\draw[thick] (9,2)--(9,3);
\vertex{(7,1)}{(0,0)}
\vertex{(7,0)}{(0,0)}
\vertex{(9,0)}{(0,0)}
\vertex{(9,1)}{(0,0)}
\vertex{(8,0)}{(0,0)}
\vertex{(8,1)}{(0,0)}
\vertex{(7,2)}{(0,0)}
\vertex{(7,3)}{(0,0)}
\vertex{(8,2)}{(0,0)}
\vertex{(8,3)}{(0,0)}
\vertex{(9,2)}{(0,0)}
\vertex{(9,3)}{(0,0)}
\end{tikzpicture}
\caption{An example of a pair of $\alpha$-cospectral graphs}
\label{fig:alpha}
\end{figure}

\bigskip
{\bf Acknowledgement.}
Much of this paper was developed while the author was visiting the Institute for Pure and Applied Mathematics (IPAM) and Brigham Young University and the author thanks these institutions for their gracious hospitality.  The author also thanks the referees of this paper for their thorough and careful reading which lead to a thorough revision and much better paper.


\begin{thebibliography}{1}
\bibitem{brouwer}
Andries Brouwer and Willem Haemers.
\newblock
\emph{Spectra of Graphs}.
\newblock Springer, 2012.

\bibitem{butler2}
Steve Butler.
\newblock \emph{Eigenvalues of 2-edge-coverings}.
\newblock Linear and Multilinear Algebra \textbf{58} (2010), 413--423.


\bibitem{butler}
Steve Butler
\newblock \emph{A note about cospectral graphs for the adjacency and normalized Laplacian matrices}.
\newblock Linear and Multilinear Algebra \textbf{58} (2010), 387--390.

\bibitem{handbook}
Steve Butler and Fan Chung.
\newblock Spectral graph theory.
\newblock in \emph{Handbook of Linear Algebra, Second Edition}, CCR Press, 2014.

\bibitem{grout}
Steve Butler and Jason Grout.
\newblock \emph{A construction of cospectral graphs for the normalized Laplacian}.
\newblock Electronic Journal of Combinatorics \textbf{18} (2011), \#P231, 20pp.

\bibitem{cavers}
Michael Cavers.
\newblock \emph{The normalized Laplacian matrix and general Randic index of graphs}.
\newblock Ph.D. Thesis, University of Regina, 2010.

\bibitem{chung}
Fan Chung.
\newblock \emph{Spectral Graph Theory, Second Edition}.
\newblock CBMS Regional Conference Series in Mathematics, 92. American Mathematical Society, 1997.

\bibitem{osborne}
Steven Osborne.
\newblock \emph{Cospectral bipartite graphs for the normalized Laplacian}.
\newblock Ph.D. Thesis, Iowa State University, 2013.

\end{thebibliography}
\end{document}